\documentclass[10pt, reqno]{amsart}
\usepackage{lipsum}
\usepackage{a4wide}
\usepackage{amssymb} 
\usepackage{amsmath}
\usepackage{amsthm} 
\usepackage{amsmath, amssymb, dsfont}
\usepackage{tikz-cd} 
\usepackage[all]{xy}
\usepackage{amscd}
\usepackage{hyperref}
\usepackage{mathrsfs}
\usepackage{enumitem}
\hypersetup{colorlinks,linkcolor={blue},citecolor={blue},urlcolor={black}}
\theoremstyle{plain}
\numberwithin{equation}{section}

\newtheorem{theorem}{Theorem}[subsection]
 
\newtheorem{lemma}[theorem]{Lemma}
\newtheorem{remark}[theorem]{Remark}
\newtheorem{proposition}[theorem]{Proposition}

\newtheorem{definition}[theorem]{Definition}

\title{Tate cohomology of unipotent cuspidal representations of finite groups of Lie type}
\author{Sabyasachi Dhar}
\begin{document}
\maketitle
\begin{abstract}
Let $p$ be a prime and $k$ be an algebraic closure of the finite field with $p$ elements. Let $G$ be a connected reductive algebraic group over $k$, and let $F:G\rightarrow G$ be a Frobenius morphism, defining an $\mathbb{F}_q$-rational structure on $G$, where $q$ is a power of $p$. Let $\ell$ be a prime different from $p$ and $\ell\ne 2$. Let $\Pi$ be an $\ell$-adic cuspidal unipotent representation of $G^{F^\ell}$ such that $\Pi\simeq \Pi^\sigma$, where $\sigma$ is a generator the Galois group ${\rm Gal}(\mathbb{F}_{q^\ell}/\mathbb{F}_q)$. In this paper, we prove the irreducibility of the zeroth Tate cohomology group $\widehat{\rm H}^0(\sigma,\mathcal{L})$ and vanishing of the first Tate cohomology group $\widehat{\rm H}^1(\sigma,\mathcal{L})$ for any $G^{F^\ell}\rtimes{\rm Gal}(\mathbb{F}_{q^\ell}/\mathbb{F}_q)$-stable lattice $\mathcal{L}$ in $\Pi$, using the multiplicity-one property of a progenerator constructed via the generalized Gelfand-Graev representations and Harish-Chandra induction.
\end{abstract}
\section{Introduction}
Let $p$ be a prime and $k=\overline{\mathbb{F}}_p$ be an algebraic closure of the finite field of size $p$. Let $G$ be a connected reductive algebraic group over $k$, and assume that $G$ is defined over the finite subfield $\mathbb{F}_q\subseteq k$, where $q$ is a power of $p$. Let $F$ be the Frobenius map on $G$, defining the $\mathbb{F}_q$-rational structure on $G$. Then $G^F$ is a finite group. Let $\ell$ be a prime different from $p$. Let $\Pi$ be an irreducible representation of $G^{F^\ell}$ defined over the maximal unramified extension $\mathbb{Q}_\ell^{\rm un}$ of $\mathbb{Q}_\ell$, such that $\Pi$ extends as a representation of $G^{F^\ell}\rtimes\langle\sigma\rangle$, where $\sigma$ is a generator of the Galois group ${\rm Gal}(\mathbb{F}_{q^\ell}/\mathbb{F}_{q})$. Then, for any $G^{F^\ell}\rtimes\langle\sigma\rangle$-stable $\mathbb{Z}_\ell^{\rm un}$-lattice $\mathcal{L}$ in $\Pi$, the Tate cohomology groups $\widehat{\rm H}^i(\sigma,\mathcal{L})$, for $i=0,1$, are $\overline{\mathbb{F}}_\ell$-representations of $G^F$. In this article, we study these Tate cohomology groups when $\Pi$ is a unipotent cuspidal representation of $G^{F^\ell}$ and $\ell\ne 2$.

This paper is broadly motivated by a conjecture of Treumann and Venkatesh (\cite[Conjecture 6.1]{TV_functoriality}) on mod-$\ell$-functoriality. Their conjectures predict that if $\rho$ is a smooth irreducible $\overline{\mathbb{F}}_\ell$-representation of $\mathcal{G}(L)$ with $\rho\simeq \rho^{\sigma}$, where $\mathcal{G}$ is a reductive algebraic group defined over a $p$-adic field $L$ and $\sigma$ an order $\ell$ automorphism of $\mathcal{G}$ defined over $L$, then the irreducible $\mathcal{G}^\sigma(L)$-subquotients of $\widehat{\rm H}^i(\sigma,\rho)$ are related to the $\mathcal{G}(F)$-representation $\rho$ via mod-$\ell$ functoriality in the sense of Langlands. In recent years, there has been considerable progress towards these conjectures; see \cite{Ronchetti_BC}, \cite{Nadimpalli_2024_GLn}, \cite{feng2023modular}, \cite{feng2024smith}, \cite{dhar2025jacquet}, \cite{cotner2026modular}, \cite{cotner2026local}.
However, the explicit computations, such as the non-vanishing of the Tate cohomology spaces and their irreducibility, are not well understood. This paper mainly aims to study the irreducibility of these Tate cohomology spaces when $L$ is a finite field, $G$ a connected reductive group defined over $L$, and $\Pi$ is a unipotent cuspidal $\mathbb{Q}_\ell^{\rm un}$-representation of $G^{F^\ell}$, where $F$ is the Frobenius endomorphism on $G$ corresponding to the $L$-structure on $G$. 

When $G$ is a classical group, if a unipotent cuspidal representation exists, then it is unique and hence invariant under the automorphism of $G$ induced by the Galois automorphism $\sigma$. Additionally, the unipotent cuspidal representations of $G^F$ are defined over $\mathbb{Q}_\ell$, and they are the building blocks of unipotent representations of $G^F$. Moreover, their mod-$\ell$ reductions are irreducible (\cite[Theorem 1.1]{Dudas_Malle_irr_mod_l}). So, the case of unipotent cuspidal representations is a natural testing ground for Treumann--Venkatesh's conjecture. Motivated by this, a recent joint work \cite[Section 7]{dhar2025jacquet} of Nadimpalli and the author proves the irreducibility of the zeroth Tate cohomology of the unipotent cuspidal representation of the symplectic group ${\rm Sp}_4(\mathbb{F}_{q^\ell})$, and shows vanishing of the first Tate cohomology group of this unipotent cuspidal representation. In this paper, we extend this work to arbitrary finite groups of Lie type having some property of the descents of $\sigma$-stable unipotent conjugacy classes. We now state the main result of this article.
\begin{theorem}\label{intro_main_thm}
Let $G$ be a connected reductive group over $k$, equipped with a Frobenius endomorphism $F$ on $G$ such that $G$ has property $\mathscr{D}$-orbit. Assume that $\ell$ and $p$ are good for $G$, and $\ell$ does not divide $p\,\lvert Z(G)^F/Z^0(G)^F\rvert$. Let $\Pi$ be a unipotent cuspidal $\mathbb{Q}_\ell^{\rm un}$-representation of $G^{F^\ell}$ such that $\Pi\simeq \Pi^\sigma$. Then, for any $\mathbb{Z}_\ell^{\rm un}[G^{F^\ell}\rtimes\langle\sigma\rangle]$-stable lattice $\mathcal{L}$ in $\Pi$, the Tate cohomology $\widehat{\rm H}^0(\sigma,\mathcal{L})$ is irreducible as a representation of $G^F$. Moreover, the first Tate cohomology group $\widehat{\rm H}^1(\sigma,\mathcal{L})$ is trivial.
\end{theorem}
Here, $Z(G)$ denotes the center of $G$, and $Z^0(G)$ is the identity component of $Z(G)$. The property ``$\mathscr{D}$-orbit" is a technical hypothesis on $G$ about the descent of $\sigma$-invariant $F^\ell$-stable unipotent conjugacy classes to the unipotent $F$-stable conjugacy classes (see subsection \ref{Descent_orbit_subsection} for detailed discussion). Most classical algebraic groups satisfy the property $\mathscr{D}$-orbit, and if $\ell$ does not divide $\lvert G^F\rvert$, then $G$ automatically satisfies the property $\mathscr{D}$-orbit. Recall that $p$ is good for $G$ if $p$ is good for each simple factor involved in $G$. In fact, when $p>5$, then $p$ is good for each simple type. 

We first prove Theorem \ref{intro_main_thm} when $G$ is simple of adjoint type and $p$ is good for $G$, using multiplicity-one property of the generalized Gelfand--Graev representations of $G^F$, as shown in \cite{Dudas_Malle_irr_mod_l}. To be more precise, given an $F$-stable unipotent element $u\in G$, let $\Gamma_u^{G^F}$ be the generalized Gelfand-Graev representation associated with $u$, obtained by induction of a certain one-dimensional representation from a unipotent $F$-stable subgroup of $G$. It only depends on the $G^F$-conjugacy classes of $u$. The construction is given in \cite{Kawanaka_GGGR_1, Kawanaka_GGGR_2, Kawanaka_GGGR_Shintani_lift} (with the assumption that $p$ is good for $G$) and in \cite{Malle_GGGR_bad_char} (with no restriction on $p$). If $u$ is the identity element, then $\Gamma_1^{G^F}$ is the regular representation of $G^F$; if $u$ is a regular unipotent element, then $\Gamma_u^{G^F}$ is the usual Gelfand-Graev representation (see \cite[Section 8.1]{Carter}). In \cite[Section 2.4]{Dudas_Malle_irr_mod_l}, Dudas--Malle constructed a progenerator for the category of $\overline{\mathbb{F}}_\ell$-representations of $G^F$, namely
$$ \mathscr{P}= \bigoplus_{(L,C)}R_L^G(\overline{\Gamma}_C^{L^F}),$$
where $(L,C)$ runs over a set of pairs consisting of a $1$-split Levi subgroup $L$ and an $F$-stable unipotent class $C$ of $G$ that is cuspidal for $L_{\rm ad}$, and $\overline{\Gamma}_C^{L^F}$ is the mod-$\ell$ reduction of $\Gamma_C^{L^F}$ which is the direct sum of $\Gamma_u^{L^F}$'s with $u$ running over a set of representatives of $G^F$-orbits in $C^F$ (see subsection \ref{pro_gen} for details). Here $R_L^G$ is the Harish-Chandra induction functor (see subsection \ref{HC_functors} for the definition), which coincides with the parabolic induction functor as $\ell \ne p$. Since $\Pi$ is cuspidal, the Tate cohomology groups $\widehat{\rm H}^i(\sigma,\mathcal{L})$ are also cuspidal (by \cite[Corollary 3.3.3]{dhar2025jacquet}). Then, using Frobenius reciprocity, we get the following identity
$$ {\rm dim}_{\overline{\mathbb{F}}_\ell}\,{\rm Hom}_{G^F}(\mathscr{P},\widehat{\rm H}^i(\sigma,\mathcal{L})) = \sum_{C} {\rm dim}_{\overline{\mathbb{F}}_\ell}\,{\rm Hom}_{G^F}(\overline{\Gamma}_C^{G^F},\widehat{\rm H}^i(\sigma,\mathcal{L})), $$
where $C$ runs over the set of $F$-stable unipotent conjugacy classes in $G$. Since $G$ is simple of adjoint type and $p$ is good for $G$, it follows from \cite[Proposition 3.1]{Dudas_Malle_irr_mod_l} that the only surviving term in the above sum corresponds to the $F$-stable unipotent support $C_\Pi$ of $\Pi$. Then, using the compatibility of Tate cohomology groups and the space of coinvariants with respect to the $F$-stable unipotent subgroup of $G$ involved in the construction of the generalized Gelfand-Graev representations (see Lemma \ref{Tate_injection}) and the fact that $G$ has the property $\mathscr{D}$-orbit, we get
$$ {\rm dim}_{\overline{\mathbb{F}}_\ell}\,{\rm Hom}_{G^F}(\mathscr{P},\widehat{\rm H}^1(\sigma,\mathcal{L})) =0, $$
and 
$$ {\rm dim}_{\overline{\mathbb{F}}_\ell}\,{\rm Hom}_{G^F}(\mathscr{P},\widehat{\rm H}^0(\sigma,\mathcal{L}))\leq 1. $$
The first identity shows that $\widehat{\rm H}^1(\sigma,\mathcal{L})$ is trivial as $\mathscr{P}$ is a progenerator. Furthermore, using the triviality of $\widehat{\rm H}^1(\sigma,\mathcal{L})$ together with a certain character identity, we see that the space $\widehat{\rm H}^0(\sigma,\mathcal{L})$ is non-zero. Then the irreducibility of $\widehat{\rm H}^0(\sigma,\mathcal{L})$ follows from the second identity and hence proves Theorem \ref{intro_main_thm} when $G$ is simple of adjoint type with the property $\mathscr{D}$-orbit, and $p$ is good for $G$ (see Theorem \ref{Tate_irr}). Now, consider the case where $G$ is not necessarily simple of adjoint type. Then the hypothesis on $\ell$ and the fact that $\Pi$ is trivial on $Z(G)^{F^\ell}$ reduce to the case where $G$ is semisimple of adjoint type. In that case, $G$ is a direct product $G_1^{n_1}\times\cdots\times G_r^{n_r}$, where each $G_i$ is a simple algebraic group of adjoint type such that $G^F=G^{F^{n_1}}\times\cdots\times G_r^{F^{n_r}}$. Moreover, the unipotent cuspidal representation $\Pi$ of $G^{F^\ell}$ is the tensor product of unipotent cuspidal $\sigma$-stable representations $\Pi_i$ of the simple factors $G_i^{F^{\ell n_i}}$. Then the result for simple algebraic groups of adjoint type $G_i$ discussed above gives the required result for $G$ (see Theorem \ref{Tate_irr_general}). 

As a possible application of Theorem \ref{intro_main_thm}, it is expected to prove Treumann--Venkatesh's functoriality conjecture for unipotent cuspidal representations of finite classical groups $G^F$ (see \cite[Theorem 7.2.3]{dhar2025jacquet} for $G={\rm Sp}_4$ and the unipotent cuspidal representation $\theta_{10}$ of ${\rm Sp}_4^F$), by comparing the character formulae mod $\ell$ for the unipotent cuspidal representations of $G^{F^\ell}$ and $G^F$. This will be addressed in future work.

The paper is organized as follows. In Section $2$, we discuss preliminaries that are essential to this paper. In Section $3$, we review the notion of unipotent cuspidal representation. Section $4$ is devoted to the discussion of generalized Gelfand-Graev representations. In Section $5$, we discuss the notion of Tate cohomology groups and set up some initial results. In Section $6$, we prove Theorem \ref{intro_main_thm}.
\vspace{0.4 cm}\\
\textbf{Acknowledgements.} The author is grateful to Santosh Nadimpalli for suggesting the problem and helpful insights. The author thanks Justin Trias for several discussions during this work. The author would like to thank Dipendra Prasad for his comments on the paper. The author also thanks the Indian
Institute of Technology Bombay for supporting this research through the Institute
Postdoctoral Fellowship.
\section{Preliminaries}
\subsection{Notation}
In this subsection, we fix some notation used throughout the paper. Let $p$ be a prime, and let $k$ be an algebraic closure of the finite field with $p$ elements. Let $G$ be a connected reductive algebraic group over $k$, and assume that $G$ is defined over the finite subfield $\mathbb{F}_q\subseteq k$. Let $F$ be the Frobenius morphism on $G$, defining the $\mathbb{F}_q$-rational structure on $G$. Then the set of $F$-fixed points $G^F=\{g\in G\mid F(g)=g\}$ is a finite group. For any positive integer $m$, the composition $F^m$ is also a Frobenius morphism on $G$, which defines an $\mathbb{F}_{q^m}$-rational structure on $G$. We denote by $Z(G)$ the center of $G$. For any algebraic group $\mathcal{G}$ over $k$, let $\mathcal{G}^0$ denote the connected component of $\mathcal{G}$ containing the identity. 

Let $\mathcal{K}$ be the maximal unramified extension of $\mathbb{Q}_\ell$ in $\overline{\mathbb{Q}}_\ell$, and let $\Lambda$ denote the ring of integers of $\mathcal{K}$. The prime $\ell$ is a uniformizer of $\Lambda$ and $\Lambda/\ell\Lambda\simeq \overline{\mathbb{F}}_\ell$, an algebraic closure of the finite field of size $\ell$. Let $H$ be a finite group. In this paper, we consider representations $(\pi,V)$ of $H$, where $V$ is a vector space over $\mathcal{K}$ or $\overline{\mathbb{F}}_\ell$, where $\ell$ is a prime different from $p$ and $\ell\ne 2$. Moreover, the representation $(\pi,V)$ is said to be $\ell$-adic (resp. $\ell$-modular) if $V$ is a vector space over the field $\mathcal{K}$ (resp. $\overline{\mathbb{F}}_\ell$). For any commutative ring $\mathcal{C}$ with unity, we denote by ${\rm Rep}_{\mathcal{C}}(H)$ the category of representations of $H$ with coefficients in $\mathcal{C}$. For any subgroup $K$ of $H$, ${\rm Ind}_K^H$ stands for the induction functor from ${\rm Rep}_\mathcal{C}(K)$ to ${\rm Rep}_{\mathcal{C}}(H)$.

Let $(\pi,V)$ be an $\ell$-adic or $\ell$-modular representation of $H$. We denote by $\chi_\pi$ the character of $(\pi,V)$. Let $\mathcal{K}[H]^H$ be the space of class functions $f:H\rightarrow \mathcal{K}$, i.e., $f(hxh^{-1})=f(x)$ for all $h,x\in H$. For two class functions $\varphi_1$ and $\varphi_2$ on $H$, let $\langle\varphi_1,\varphi_2\rangle$ be the scalar defined as 
$$ \langle\varphi_1,\varphi_2\rangle := \frac{1}{\lvert H\rvert}\sum_{h\in H}\varphi_1(h)\,\varphi_2(h^{-1}). $$
In particular, the character $\chi_\pi$ of the $\ell$-adic representation $(\pi,V)$ is a class function on $H$. For two $\ell$-adic representations $\pi$ and $\rho$ of $H$, we have the following identity
$$ \langle\chi_\pi,\chi_\rho\rangle = {\rm dim}_{\mathcal{K}}\,{\rm Hom}_H(\pi,\rho). $$
For a virtual representation $\xi$ of $H$, we also denote its character by $\chi_{\xi}$. If $\ell$ does not divide $\lvert H\rvert$, the scalar product defined above is valid for $\overline{\mathbb{F}}_\ell$-valued class functions on $H$.
\subsection{Harish-Chandra induction and restriction}\label{HC_functors}
In this subsection, we briefly discuss the Harish-Chandra induction and restriction functors. For a reference, see \cite[Chapter 5]{Digne_Michel_book}.
\subsubsection{}
Let $G$ be a connected reductive group defined over $\mathbb{F}_q$, and let $F$ be the corresponding Frobenius morphism on $G$. Let $\mathcal{C}$ be a commutative ring with unity in which the prime $p$ is invertible. Let $P$ be an $F$-stable parabolic subgroup of $G$ and $L$ an $F$-stable Levi subgroup of $P$, so that we have an $F$-stable Levi decomposition $P=L\ltimes U$. 
\begin{enumerate}
\item The {\it Harish-Chandra induction} is the functor 
$$ R_{L\subseteq P}^G:{\rm Rep}_{\mathcal{C}}(L^F)\rightarrow {\rm Rep}_{\mathcal{C}}(G^F) $$
defined by $R_{L\subseteq P}^G:={\rm Ind}_{P^F}^{G^F}\circ {\rm Inf}_{L^F}^{P^F}$, where the inflation ${\rm Inf}_{L^F}^{P^F}$ is the lifting through the quotient $P^F\rightarrow L^F$.
\item For $V\in {\rm Rep}_{\mathcal{C}}(P^F)$, let $r_{U^F}(V) = V^{U^F}$ the space of fixed points under $U^F$. Then the {\it Harish-Chandra restriction} is the functor 
$$ ^*R_{L\subseteq P}^G:{\rm Rep}_{\mathcal{C}}(G^F)\rightarrow {\rm Rep}_{\mathcal{C}}(L^F) $$
defined by $^*R_{L\subseteq P}^G:=r_{U^F}\circ {\rm Res}_{P^F}^{G^F}$.
\end{enumerate}
The functors $R_{L\subseteq P}^G$ and $^*R_{L\subseteq P}^G$ are independent of the choice the parabolic subgroup $P$ (see \cite[Theorem 5.3.1]{Digne_Michel_book} for a proof). So, we will use the notations $R_L^G$ and $^*R_L^G$. Now, let us highlight some of their properties:
\vspace{0.1 cm}\\
(a) The functors $R_L^G$ and $^*R_L^G$ are exact.
\vspace{0.1 cm}\\
(b) For $\pi\in {\rm Rep}_{\mathcal{C}}(G^F)$ and $\rho\in {\rm Rep}_{\mathcal{C}}(L^F)$, we have the following isomorphims
$$ {\rm Hom}_{G^F}(\pi,R_L^G(\rho)) \simeq {\rm Hom}_{L^F}(^*R_L^G(\pi),\rho), $$
and 
$$ {\rm Hom}_{G^F}(R_L^G(\rho),\pi) \simeq {\rm Hom}_{L^F}(\rho,\,^*R_L^G(\pi)). $$
In other words, $R_L^G$ and $^*R_L^G$ are biadjoint. 

\subsubsection{}
Let $(\pi,V)$ be an $R$-representation of $U^F$. For any character $\chi:U^F\rightarrow R^\times$, let $V(U^F,\chi)$ be the space generated by the set of vectors $\{\pi(u)v-\chi(u)v\mid u\in U^F,v\in V\}$. Since $U^F$ is a $p$-group and $\ell\ne p$, we have the following decomposition 
$$ V= V^{U^F,\chi}\oplus V(U^F,\chi), $$ 
where $V^{U^F,\chi}:=\{v\in V\mid \pi(u)v=\chi(u)v,\,\,\forall u\in U^F\}$. Let $V_{U^F,\chi}$ denote the quotient space $V/V(U^F,\chi)$. This is the maximal quotient of $V$ on which $U^F$ acts by $\chi$ and is said to be the space of $(U^F,\chi)$-coinvariants of $V$ or simply the space of coinvariants of $V$ when the pair $(U^F,\chi)$ is clear from the context. From the above decomposition, it follows that 
$$ V^{U^F,\chi} = V_{U^F,\chi}. $$
Thus, the Harish-Chandra restriction $^*R_L^G(\rho)$ for $\rho$ an $R$-representation of $G^F$, is the space of $(U^F,1)$-coinvariants $\rho_{U^F}:=\rho_{U^F,1}$. 
\begin{definition}\normalfont
An $R$-representation $(\pi,V)$ of $G^F$ is called {\it cuspidal} if $^*R_L^G(V)=0$ for all $F$-stable Levi subgroups $L$ of all proper $F$-stable parabolic subgroups $P$ of $G$.
\end{definition}
\subsubsection{}\label{DL_ind}
In \cite{Lusztig_unipotent_class}, Lusztig constructed the functor $R_L^G$ to the case where $L$ is an $F$-stable Levi subgroup of $G$ which is not the Levi subgroup of any $F$-stable parabolic subgroup of $G$. Let $P=L\ltimes U$ be a Levi decomposition of a parabolic subgroup of $G$ where $L$ is $F$-stable. Let us consider the variety
$$ X_U:=\{g\in G\mid g^{-1}F(g)\in F(U)\}/(U\cap F(U)). $$
Put ${\rm H}_c^*(X_U)=\sum_{i\geq 0}(-1)^i{\rm H}_c^i(X_U)$, where ${\rm H}^i_c(X_U)$ is the $\ell$-adic cohomology of $X_U$ with compact support (see \cite[Rapport 2.10]{cohomology} for details). Consider the action of $G^F\times L^F$ on $X_U$ given by 
$$ (g,m)\cdot x = g^{-1}xm^{-1}. $$ 
This action induces an action of $G^F\times L^F$ on each ${\rm H}^i_c(X_U)$, and hence on ${\rm H}_c^*(X_U)$. The Lusztig functor is a functor 
$$ R_{L\subseteq P}^G:{\rm Rep}_{\mathcal{C}}(L^F)\rightarrow {\rm Rep}_{\mathcal{C}}(G^F) $$
defined as $R_{L\subseteq P}^G(\rho) = {\rm H}_c^*(X_U)_\rho$, where ${\rm H}_c^*(X_U)_\rho$ is the $\rho$-isotypic component of the $G^F\times L^F$-module ${\rm H}_c^*(X_U)$. As in the previous case, the functor $R_{L\subseteq P}^G$ is independent of the choice of the parabolic subgroup $P$. So, we will use the notation $R_L^G$ for the Lusztig functor.
\begin{definition}\normalfont
When $L=T$ an $F$-stable maximal torus in $G$ and $\theta$ is an one-dimensional character of $T^F$, then $R_T^G(\theta)$ is called a {\it Deligne--Lusztig character}.
\end{definition}
These characters were introduced by Deligne--Lusztig in \cite{DL_finite_groups}. In general, $R_T^G(\theta)$ is not irreducible, and it is irreducible if and only if $\theta$ is in ``general position", i.e., $\theta^w\ne \theta$ for all $w\in N_G(T)^F\setminus T^F$. In particular, when $\theta=1$, the trivial character of $T^F$, then $R_T^G(1)$ is not irreducible.

\subsection{Integral structure and mod-$\ell$ reduction}
Let $G$ be a group. An $\ell$-adic representation $(\pi,V)$ of $G$ is said to be integral if it has finite length, and there exists a $\Lambda[G]$-stable lattice $\mathcal{L}$ in $V$. Then we have an $\ell$-modular representation $\mathcal{L}\otimes_\Lambda\overline{\mathbb{F}}_\ell$ of $G$, which is of finite length. The semisimplification of the $\ell$-modular representation $\mathcal{L}\otimes_\Lambda
\overline{\mathbb{F}}_\ell$ is independent of the choice of the $\Lambda[G]$-invariant lattice $\mathcal{L}$ in $V$ and is called the mod-$\ell$ reduction of $\pi$, denoted by $r_\ell(\pi)$. 

When $G$ is a finite group, then $(\pi,V)$ always admits a $\Lambda[G]$-stable lattice in $V$. If $H$ is a subgroup of $G$, and $\tau$ is an $\ell$-adic representation of $H$ admitting a $\Lambda[H]$-stable lattice $\mathcal{L}$, the induced module ${\rm Ind}_H^G(\mathcal{L})$ is also a $\Lambda[G]$-stable lattice in ${\rm Ind}_H^G(\tau)$. If $G$ is a connected reductive group over $k$ with a Frobenius morphism $F$ on $G$, and if $P=L\ltimes U$ is an $F$-stable parabolic subgroup of $G$ with $L$ an $F$-stable Levi subgroup and $U$ an $F$-stable unipotent subgroup of $P$, then for any $\ell$-adic representations $\rho$ and $\pi$ of $L^F$ and $G^F$ respectively, we have
$$ r_\ell(R_L^G(\rho)) = R_L^G(r_\ell(\rho)) $$
and
$$ r_\ell(^*R_L^G(\pi)) =\, ^*R_L^G(r_\ell(\pi)). $$
Hence, Harish-Chandra induction and restriction commute with mod-$\ell$ reduction (see \cite[Chapter I, Section 9.3]{Vigneras_mod_l_book}).

\section{Unipotent cuspidal representation}
In this part, we recall the notion of unipotent cuspidal representations and their properties essential to this article. We mainly follow the exposition of \cite{Digne_Michel_book} and \cite{Dudas_Malle_irr_mod_l}.
\subsection{}
Let $G$ be a connected reductive group over $k$, and let $F$ be a Frobenius endomorphism on $G$. For an $F$-stable maximal torus $T$ in $G$ and a character $\theta:T^F\rightarrow \mathcal{K}^\times$, the associated Deligne--Lusztig character is denoted by $R_T^G(\theta)$ (see subsection \ref{DL_ind} for the definition).  
\begin{definition}\normalfont
An irreducible $\ell$-adic representation $\pi$ of $G^F$ is said to be {\it unipotent} if there exists an $F$-stable maximal torus $T$ of $G$ such that $\pi$ occurs in $R_T^G(1)$, i.e.,
$$\langle \chi_\pi, R_{T}^G(1)\rangle_{G^F}\ne 0, $$ where $1$ is the trivial representation of $T^F$. 
\end{definition}
\subsection{Unipotent support}
Let $\rho$ be an irreducible $\ell$-adic representation of $G^F$. Let $C$ be an $F$-stable unipotent class of $G$. Define the average value of $\rho$ on $C^F$ as follows:
$$ {\rm AV}(C,\rho):= \dfrac{1}{\lvert C^F\rvert}\sum_{u\in C^F} \chi_\rho(u). $$
The unipotent class $C$ is called a {\it unipotent support} of $\rho$ if $C$ has maximal dimension among $F$-stable unipotent classes for which ${\rm AV}(C,\rho)\ne 0$. The following theorem, due to Geck (\cite[Theorem 1.4]{Geck_unicity_support}), gives uniqueness of unipotent support of $\rho$.
\begin{theorem}
When $p$ is a good prime for $G$, then an irreducible $\ell$-adic representation $\rho$ of $G^F$ has a unique unipotent support.    
\end{theorem}
We denote by $C_\rho$ the unique unipotent support of an irreducible representation $\rho$ of $G^F$. An $F$-stable unipotent class $C$ of $G$ is said to be {\it non-cuspidal} if there exists a proper $1$-split Levi subgroup $L$ of $G$ such that $C\cap L\ne \emptyset$, and for all $u\in C\cap L$, the natural map 
$$ C_{L}(u)/C_{L}(u)^0\longrightarrow C_G(u)/C_G(u)^0 $$
is an isomorphism, where for any algebraic group $\mathcal{G}$, the symbol $\mathcal{G}^0$ stands for the connected component of $\mathcal{G}$ containing the identity element of $\mathcal{G}$. Here, a $1$-split Levi subgroup of $G$ is by definition an $F$-stable Levi subgroup of an $F$-stable parabolic subgroup of $G$. If no such proper Levi subgroup exists, the unipotent class $C$ is called {\it cuspidal}.

In \cite{Dudas_Malle_irr_mod_l}, Dudas--Malle construct a progenerator $\mathscr{P}$ of the category of both $\ell$-adic and $\ell$-modular representations of $G^F$ using generalized Gelfand-Graev representations associated to suitably chosen unipotent classes. They then prove that the cuspidal unipotent characters of $G^F$ occur with multiplicity one in $\mathscr{P}$. This multiplicity one result is very crucial in deducing the irreducibility of the mod-$\ell$ reduction of unipotent cuspidal representations of $G^F$ under a restrictive hypothesis on $p$, $\ell$ and the groups involved (see \cite[Theorem 1]{Dudas_Malle_irr_mod_l}). We state the irreducibility result for future reference.
\begin{theorem}\label{mod_l_irr_unip_cusp}
Let $G$ be a connected reductive group over $k$, equipped with a Frobenius morphism $F:G\rightarrow G$. Assume that the primes $\ell$ and $p$ are good for $G$, and $\ell$ does not divide $p\,\lvert Z(G)^F/Z^0(G)^F\rvert$. Then, the mod-$\ell$ reduction of any unipotent cuspidal $\ell$-adic representation of $G^F$ is irreducible.
\end{theorem}

\section{Generalized Gelfand-Graev representation}
In this section, we recall the construction of the generalized Gelfand--Graev representation (``GGGR" for short) for a connected reductive algebraic group $G$ defined over $\mathbb{F}_q$, with corresponding Frobenius morphism $F:G\rightarrow G$. An ordinary Gelfand--Graev representation of $G^F$ is associated with a regular nilpotent ${\rm Ad}(G)$-orbit in $\mathfrak{g}={\rm Lie}(G)$. A generalized Gelfand-Graev representation of $G^F$ is associated with a unipotent element of $G$, and the construction of such representations is due to Kawanaka. Throughout this section, we assume that $p$ is good for $G$. We mainly follow the exposition of \cite{Geck_unicity_support}.
\subsection{}
Let us fix an $F$-stable maximal torus $T\subseteq G$ and an $F$-stable Borel subgroup $B\subseteq G$ containing $T$. Let $\Phi$ be the set of roots of $G$ with respect to $T$, and let $\Delta$ be the set of simple roots of $G$ determined by $B$. Let $\mathfrak{g}$ be the Lie algebra of $G$. Then $G$ acts on $\mathfrak{g}$ via the adjoint representation ${\rm Ad}:G\rightarrow {\rm GL}(\mathfrak{g})$. We have a direct sum decomposition
$$ \mathfrak{g}=\mathfrak{t}\oplus
\bigoplus_{\alpha\in \Phi}\mathfrak{g}_\alpha, $$
where $\mathfrak{t}$ denotes the Lie algebra of $T$ and $\mathfrak{g}_\alpha=\{X\in\mathfrak{g}\mid {\rm Ad}(t)(X)=\alpha(t)X,\,\forall t\in T\}$. For each $\alpha\in\Phi$, we have a one-parameter subgroup $x_\alpha:\mathbb{G}_a\rightarrow G$, whose image is denoted by $U_\alpha$. The subspace $\mathfrak{g}_\alpha$ is the image of the differential $dx_\alpha:\mathbb{G}_a\rightarrow \mathfrak{g}$, and we set $X_\alpha=dx_\alpha(1)\in \mathfrak{g}_\alpha$. Then $X_\alpha\ne 0$ and $\mathfrak{g}_\alpha$ is a one-dimensional subspace generated by $X_\alpha$.
\subsection{}
Fix a unipotent element $u\in G^F$, and let $d:\Phi\rightarrow \mathbb{Z}$ be the associated weighted Dynkin diagram (see \cite[Section 2]{Geck_unicity_support} for the definition). For $i\in\mathbb{Z}$, set $\Phi_i=\{\alpha\in\Phi: d(\alpha)=i\}$ and define
$$ \mathfrak{g}(i):=\begin{cases}
    \bigoplus_{\alpha\in \Phi_i}\mathfrak{g}_\alpha & \text{if } i \ne 0, \\
\mathfrak{t}\oplus\bigoplus_{\alpha\in\Phi_0}
\mathfrak{g}_\alpha  & \text{if } i=0.
\end{cases}$$
Then we have the following decomposition
$$ \mathfrak{g} = \bigoplus_{i\in\mathbb{Z}} \mathfrak{g}(i), $$
Consider the parabolic subgroup 
$P:=\langle\, T,U_\alpha\mid \alpha\in \Phi\,\,\forall i\geq 0\,\rangle$ with the unipotent radical $U_1:=\langle\,U_\alpha\mid \alpha\in\Phi_i\,\,\forall i\geq 1\,\rangle$ and Levi subgroup $L:=\langle\,T,U_\alpha\mid\alpha\in \Phi_0\,\rangle$.  
For each $i\geq 1$, let $U_i$ be the subgroup corresponding to the Lie algebra 
$$ \bigoplus_{j\geq i} \mathfrak{g}(i). $$
Then we have a decreasing chain of subgroups $U_1\supseteq U_2\supseteq U_3\supseteq \cdots$, and each $U_i$ is a normal subgroup of $P$ such that the quotient group $U_i/U_{i+1}$ is abelian. Note that the subgroups $P$, $U_i\,(i\geq 1)$ of $G$ are $F$-stable, and each unipotent group $U_i^F$ is a $p$-group. Also note that the subspaces $\mathfrak{g}(i)\,\,(i\geq 0)$ are $F$-stable. 

Consider a map
$f_{\rm spr}:U_1\rightarrow \mathfrak{g}(1)\oplus \mathfrak{g}(2)$ defined as follows. Let $u\in U_1$. The element $u$ has a unique expression
$$ u=\prod_{\alpha\in\Phi_i\,\text{for}\,i\geq 1} x_\alpha(u_\alpha)\,\,\,\,\,\, (u_\alpha\in \overline{\mathbb{F}}_q). $$
Then 
$$ f_{\rm spr}(u) :=\sum_{\alpha\in\Phi_1\cup\Phi_2}u_\alpha X_\alpha. $$
The map $f_{\rm spr}$ has the following properties, which we state as a lemma (see \cite[Lemma 2.7]{Malle_GGGR_bad_char} for a proof).
\begin{lemma}\label{springer_map_lemma}
Let $x,y\in U_1$. Then the following holds.
\begin{enumerate}
    \item If $x\in U_2$ or $y\in U_2$, then $f_{\rm spr}(xy)=f_{\rm spr}(x)+f_{\rm spr}(y)$.
    \item $f_{\rm spr}(xyx^{-1}y^{-1})- [f_{\rm spr}(x),f_{\rm spr}(y)] \in \bigoplus_{i\geq 3}\mathfrak{g}(i)$.
    \item $f_{\rm spr}(F(x)) = F(f_{\rm spr}(x))$.
\end{enumerate}
\end{lemma}
\subsection{}\label{GGGR_construction}
Fix a non-trivial additive character $\psi:\mathbb{F}_q\rightarrow\Lambda^\times$. Consider a linear map $\lambda:\mathfrak{g}(2)\rightarrow \mathbb{G}_a$ defined over $\mathbb{F}_q$, i.e., $\lambda(F(z)) = \lambda(z)^q$ for all $z\in \mathfrak{g}(2)$, such that $\lambda$ is in ``sufficiently general position" (see \cite[Definition 3.4]{Geck_unicity_support} for the definition). Lemma \ref{springer_map_lemma}($1$) shows that there is a group homomorphism $U_2\rightarrow\mathbb{G}_a$, given by $x\mapsto \lambda(f_{\rm spr}(x))$. Using Lemma \ref{springer_map_lemma}($3$), we also obtain a group homomorphism
$\psi_u:U_2^F\rightarrow\Lambda^\times$, defined by
$$ \psi_u(x) := \psi(\lambda(f_{\rm spr}(x))), $$
for all $x\in U_2^F$. If $\mathfrak{g}(1)=\{0\}$, then the GGGR corresponding to $u$ is given by the induced representation 
$$ \Gamma_u^{G^F}:={\rm Ind}_{U_2^F}^{G^F}(\psi_u). $$
Now consider the case where
$\mathfrak{g}(1)\ne \{0\}$. Since $[\mathfrak{g}(1),\mathfrak{g}(1)]\subseteq \mathfrak{g}(2)$, there exists a well-defined skew-symmetric bilinear form 
$$ \mathfrak{g}(1)\times\mathfrak{g}(1)\rightarrow \mathbb{G}_a,\,\,\, (y,z)\mapsto \lambda([y,z]) $$
such that the radical of this bilinear form is trivial. Now, we choose an $F$-stable Lagrangian subspace in $\mathfrak{g}(1)$ and pull back this subspace to an $F$-stable subgroup $U_{1.5}\subseteq U_1$ via the map $f_{\rm spr}$. By Lemma \ref{springer_map_lemma}($2$), we get that the kernel ${\rm Ker}(\psi_u)$ is normal in $U_{1.5}^F$, and the quotient $U_{1.5}^F/{\rm Ker}(\psi_u)$ is an abelian $p$-group. Then $\psi_u$ can be extended to a character $\widetilde{\psi}_u:U_{1.5}^F\rightarrow \Lambda^\times$.

\begin{definition}\normalfont
The induced representation $\Gamma_u^{G^F}:= {\rm Ind}_{U_{1.5}^F}^{G^F}(\widetilde{\psi}_u)$ is called the {\it generalized Gelfand--Graev representation} of $\textbf{G}^F$ associated with $u$.
\end{definition}
We list some elementary properties of GGGR. 
\begin{enumerate}
    \item $\Gamma_u^{G^F}$ is a projective $\Lambda[G^F]$-lattice.
    \item If $u'$ is a unipotent element such that $u'$ is $G^F$-conjugate to $u$, then $\Gamma_u^{G^F} \simeq \Gamma_{u'}^{G^F}$.
\end{enumerate}
Let $\gamma_u^{G^F}$ denote the character of the $\mathcal{K}$-representation $\Gamma_u^{G^F}\otimes_\Lambda \mathcal{K}$ and is called the generalised Gelfand--Graev character associated with $u$. The character $\gamma_{u}^{G^F}$ depends on the $G^F$-conjugacy class of $u$. When $u$ is a regular unipotent element, $\gamma_u^{G^F}$ is the ordinary Gelfand--Graev character. We denote by $\overline{\Gamma}_u^{G^F}$ the mod-$\ell$ reduction $\Gamma_u^{G^F}\otimes_\Lambda \overline{\mathbb{F}}_\ell$, and the corresponding mod-$\ell$ character is denoted by $\overline{\gamma}_u^{G^F}$.

Let $\rho$ be an irreducible $\mathcal{K}$-representation of $G^F$. Let $\rho^\ast$ be the Alvis--Curtis dual of $\rho$. The following theorem, due to Lusztig, gives a condition on the unipotent support of a character to occur in a generalised Gelfand--Graev character.
\begin{theorem}[Lusztig]
Let $\rho$ be an irreducible $\mathcal{K}$-representation of $G^F$ with character $\chi_\rho$. Then:
\begin{enumerate}
    \item there exists $u\in C_{\rho^\ast}^F$ such that $\langle\gamma_u^{G^F},\chi_\rho\rangle \ne 0$;
    \item if $C$ is an $F$-stable unipotent conjugacy class of $G$ such that ${\rm dim}(C)> {\rm dim}(C_{\rho^\ast})$, then $\langle\gamma_u^{G^F},\chi_\rho\rangle=0$ for all $u\in C^F$.
\end{enumerate}
\end{theorem}
\subsection{Progenerator}\label{pro_gen}
We first recall the definition of a progenerator. Given an abelian category $\mathcal{A}$, an object $\mathcal{P}$ in $\mathcal{A}$ is a called a progenerator if the following holds:
\begin{enumerate}
\item $\mathcal{P}$ is a projective object of $\mathcal{A}$;
\vspace{0.1 cm}
\item ${\rm Hom}(\mathcal{P},M)$ is non-zero for every non-zero object $M$ of $\mathcal{A}$;
\vspace{0.1 cm}
\item for any family $M_i$ of objects in $\mathscr{A}$ indexed by a set $I$, one has an isomorphism
$$ {\rm Hom}(\mathcal{P},\bigoplus_{i\in I} M_i) \simeq \bigoplus_{i\in I}{\rm Hom}(\mathcal{P},M_i). $$
\end{enumerate}
Let $G$ be a connected reductive group with Frobenius $F$. Recall that the set of $\ell$-modular representations of $G^F$ forms an abelian category. In \cite[Corollary 2.3]{Dudas_Malle_irr_mod_l}, Dudas--Malle gave a progenerator of this category using generalised Gelfand--Graev representations. Let $C$ be an $F$-stable cuspidal unipotent class of $G$. Define 
$$ \Gamma_C^{G^F} := \bigoplus_{u\in C^F/G^F}\Gamma_u^{G^F},$$
where $C^F/G^F$ is the set of distinct $G^F$-orbits in $C^F$, and $u$ varies over a set of representatives from each of these orbits. We denote by $\overline{\Gamma}_C^{G^F}$ the mod-$\ell$ reduction $\Gamma_C^{G^F}\otimes_\Lambda
\overline{\mathbb{F}}_\ell$. Let $\mathscr{L}$ be the set of pairs $(L,C)$ such that $L$ is a $1$-split Levi subgroup of $G$, and $C$ is an $F$-stable unipotent class of $L$ which is cuspidal for $L_{\rm ad}$. 
\begin{theorem}
Let $G$ be a connected reductive group over $k$ with a Frobenius morphism $F$ on $G$, and assume that $p$ is good for $G$. Then the module 
\begin{equation}\label{progenerator}
\mathscr{P} = \bigoplus_{(L,C)\in\mathscr{L}}R_L^G
(\overline{\Gamma}_C^{G^F})
\end{equation}
is a progenerator of the category of $\ell$-modular representations of $G^F$.
\end{theorem}

\section{Tate cohomology}
This part is devoted to the discussion of Tate cohomology groups of various $\Lambda[\Gamma]$-modules, where $\Gamma$ is a cyclic group of order $\ell$.
\subsection{}
Let $\sigma$ be a generator of the cyclic group $\Gamma$. Set ${\rm Nr}={\rm id}+\sigma+\cdots+\sigma^{\ell-1}$. Let $M$ be an $\Lambda[\sigma]$-module. Then the Tate cohomology groups of $M$ are defined as follows:
$$ \widehat{\rm H}^0(\sigma,M) := \dfrac{{\rm Ker}({\rm id}-\sigma)}{{\rm Im}({\rm Nr})}, \,\,\, \widehat{\rm H}^1(\sigma,M) := \dfrac{{\rm Ker}({\rm Nr})}{{\rm Im}({\rm id}-\sigma)}. $$
Note that each $\widehat{\rm H}^i(\sigma,M)$ is a vector space over $\overline{\mathbb{F}}_\ell$. If $H$ is a group endowed with an action of $\sigma$, and $M$ is an $\Lambda[H\rtimes \sigma]$-module, then each Tate cohomology group $\widehat{\rm H}^i(\sigma,M)$ carries an action of the fixed-point subgroup $H^\sigma$, giving two $\ell$-modular representations of $H^\sigma$. In this article, we deal with the case where $H=G^{F^\ell}$ for a connected reductive group $G$ with Frobenius map $F$, $\Gamma={\rm Gal}(\mathbb{F}_{q^\ell}/\mathbb{F}_q)$, and $\sigma$ is the Frobenius automorphism $x\mapsto x^q$. In this case, the fixed-point subgroup $H^\sigma$ is equal to $G^F$.

We now recall a result due to Treumann--Venkatesh (\cite[Proposition 3.3.1]{TV_functoriality}), which gives the compatibility of Tate cohomology groups with induction. This result is crucial in deducing our main theorems.
\begin{proposition}\label{TV_isom}
Let $H$ be a finite group, and $\sigma$ be an automorphism of $H$ of prime order $\ell$. Let $K$ be a $\sigma$-stable subgroup of $H$. Then, for any $\Lambda[K\rtimes \sigma]$-module $V$, the restriction map 
$$ \big({\rm Ind}_K^H(v)\big)^\sigma \rightarrow {\rm Ind}_{K^\sigma}^{H^\sigma}(\widehat{\rm H}^i(\sigma, V)) $$
$$ \varphi\mapsto \varphi|_{H^\sigma} $$
induces an isomorphism of $H^\sigma$-representations
$$ \widehat{\rm H}^i\big(\sigma, {\rm ind}_K^H(V)\big) \simeq 
{\rm Ind}_{K^\sigma}^{H^\sigma}(\widehat{\rm H}^i(\sigma,V)). $$ 
\end{proposition}
Next, we recall a couple of results established in \cite{dhar2025jacquet}. The first result shows that Tate cohomology preserves cuspidality. Although the result is true for $p$-adic groups as well, but we state it only for finite groups of Lie type.
\begin{proposition}\cite[Corollary 3.3.3]{dhar2025jacquet}\label{Tate_cuspidal}
Let $G$ be a connected reductive algebraic group with the Frobenius map $F$. Let $\tau$ be a cuspidal $\ell$-adic representation of $G^{F^\ell}$ such that $\tau$ extends as a representation of $G^{F^\ell}\rtimes\sigma$. Then, for any $\Lambda[G^{F^\ell}\rtimes\sigma]$-stable lattice $\mathcal{L}$ in $\tau$, the Tate cohomology groups $\widehat{\rm H}^i(\sigma,\mathcal{L})$ are cuspidal.
\end{proposition}
The next result gives
the trace of the $G^\sigma$-action on the Tate cohomology with respect to
the action of $\sigma$. For $x\in \Lambda$, we denote by $\overline{x}$ its image
in $\overline{\mathbb{F}}_\ell$.
\begin{lemma}\cite[Lemma 5.2.1]{dhar2025jacquet}\label{char_iden}
Let $J$ be a cyclic group of prime order $\ell$, and let $G$ be a finite group with an action of $J$. Let $M$ be a 
$\Lambda[G\rtimes J]$-module such that $M$ is a free $\Lambda$-module of 
finite rank. Assume that $\widehat{{\rm H}}^1(J, M) =0$. Then
$$ {\rm Tr}(g|_{\widehat{{\rm H}}^0(J, M)})=\overline{{\rm Tr}(g|_M)}, $$ 
for all $g\in G^J$.
\end{lemma}

\subsection{Tate cohomology of generalized Gelfand-Graev representations}
Let $G$ be a connected reductive algebraic group defined over $\mathbb{F}_q$, with the corresponding Frobenius morphism $F:G\rightarrow G$. The fixed point subgroups under $F$ and $F^\ell$ are the finite groups $G^F= G(\mathbb{F}_q)$ and $G^{F^\ell}=G(\mathbb{F}_{q^\ell})$. Let $\sigma$ be a generator of the Galois group ${\rm Gal}(\mathbb{F}_{q^\ell}/\mathbb{F}_q)$. Let $u\in G^F$ be a unipotent element. Then $u$ is also a unipotent element of $G^{F^{\ell}}$. Fix a non-trivial additive character $\psi:\mathbb{F}_q\rightarrow\Lambda^\times$, and let $\eta$ denote the composition $\psi\circ{\rm Tr}_{\mathbb{F}_{q^\ell}/\mathbb{F}_q}$. Recall that the linear map $\lambda:\mathfrak{g}(2)\rightarrow\mathbb{G}_a$ is compatible with the Frobenius maps $F$ and $F^\ell$. In particular, we get a linear map $\lambda:\mathfrak{g}(2)^{F^\ell}\rightarrow\mathbb{F}_{q^\ell}$ which is also $\sigma$-invariant, i.e., $\lambda(\sigma \cdot X)=\sigma(\lambda(X))$ for all $X\in \mathfrak{g}(2)^{F^\ell}$. In view of Lemma \ref{springer_map_lemma}($3$), the map $f_{\rm spr}$ is also $\sigma$-invariant. Then the character $\eta_u:U_2^{F^\ell}\rightarrow\Lambda^\times$, defined by 
$$ \eta_u(x) := \eta(\lambda(f_{\rm spr}(x))),\,\,\,\forall x\in U_2^{F^\ell} $$
is $\sigma$-invariant, i.e., $\eta_u^\sigma=\eta_u$. Now, consider the following generalised Gelfand-Graev representations:
$$ \Gamma_u^{G^F} = {\rm Ind}_{U_{1.5}^F}^{G^F}(\widetilde{\psi}_u),\,\,\,\,\,\, \Gamma_u^{G^{F^\ell}} ={\rm Ind}_{U_{1.5}^{F^\ell}}^{G^{F^\ell}}(\widetilde{\eta}_u). $$
The representation $\Gamma_u^{G^{F^\ell}}$ is endowed with an action of $\sigma$, given by 
$$ (\sigma f)(g):=f(\sigma^{-1}(g)),\,\,\,\forall g\in G^{F^\ell}. $$
\begin{lemma}\label{Tate_GGGR}
Let $u\in G^F$ be a unipotent element. Then we have the $G^F$-isomorphisms
$$ \widehat{\rm H}^0(\sigma, \Gamma_u^{G^{F^\ell}}) \simeq \big(\overline{\Gamma}_u^{G^F}\big)^{(\ell)}. $$
Moreover, the first Tate cohomology group $\widehat{\rm H}^1(\sigma,\Gamma_u^{G^{F^\ell}})$ is trivial.
\end{lemma}
\begin{proof}
Since $U_{1.5}^{F^\ell}$ is a $\sigma$-stable subgroup of $G^{F^\ell}$, we have the following exact sequence of non-abelian cohomology sets:
$$ 1\longrightarrow U_{1.5}^F\longrightarrow G^F\longrightarrow \big(U_{1.5}^{F^\ell}\backslash G^{F^\ell}\big)^\sigma\longrightarrow {\rm H}^1(\sigma,U_{1.5}^{F^\ell}). $$
Since $U_{1.5}^{F^\ell}$ is a finite $p$-group and $\ell\ne p$, the cohomology set ${\rm H}^1(\sigma,U_{1.5}^{F^\ell})$ is trivial. From the preceding exact sequence, we get the equality of sets
$$ \big(U_{1.5}^{F^\ell}\backslash G^{F^\ell}\big)^\sigma = U_{1.5}^F\backslash G^F. $$
Using Theorem \ref{TV_isom}, we get the following isomorphism of $G^F$-representations
\begin{equation}\label{Tate_isom_2}
\widehat{\rm H}^i(\sigma,\Gamma_u^{G^{F^\ell}})\simeq {\rm Ind}_{U_{1.5}^F}^{G^F}(\widehat{\rm H}^i(\sigma, \widetilde{\eta}_u)).
\end{equation}
For $x\in U_{1.5}^F$, we see that $F(\lambda(f_{\rm spr}(x)))=\lambda(f_{\rm spr}(x))$, i.e., $\lambda(f_{\rm spr}(x))\in\mathbb{F}_q$. Hence, $\widetilde{\eta}_u(x)=(\widetilde{\psi}_u(x))^\ell$. Thus, we get that $\widehat{\rm H}^0(\sigma,\widetilde{\eta}_u) \simeq \overline{\widetilde{\psi}_u}^{\ell}$, where $\overline{\widetilde{\psi}_u}$ is the mod-$\ell$ reduction of $\widetilde{\psi}_u$. Also note that $\widehat{\rm H}^1(\sigma,\widetilde{\eta}_u)
|_{U_{2}^F}=\widehat{\rm H}^1(\sigma,\widetilde{\eta}_u|
_{U_{2}^{F^\ell}})=\widehat{\rm H}^1(\sigma,\eta_u) = 0$. Then it follows from the isomorphism (\ref{Tate_isom_2}) that $\widehat{\rm H}^1(\sigma,\Gamma_u^{G^{F^\ell}})$ is trivial, and 
$$ \widehat{\rm H}^0(\sigma,\Gamma_u^{G^{F^\ell}}) \simeq (\overline{\Gamma}_u^{G^F})^{(\ell)} $$
as $G^F$-representations. Hence, the lemma follows.
\end{proof}
We conclude this section with the following lemma on the compatibility of ${\rm Gal}(\mathbb{F}_{q^\ell}/\mathbb{F}_q)$-Tate cohomology groups with the space of $(U_{1.5}^F,\widetilde{\psi}_u)$-coinvariants.
\begin{lemma}\label{Tate_injection}
Let $u$ be a unipotent element in $G^F$. Then, for any $[G^{F^\ell}\rtimes\sigma]$-module $\mathcal{L}$, the natural map 
$$ \widehat{\rm H}^i(\sigma,\mathcal{L})_{{U_{1.5}^{F},
\overline{\widetilde{\psi}}_u}}\longrightarrow \widehat{\rm H}^i(\sigma,\mathcal{L}_{U_{1.5}^{F^\ell},\widetilde{\eta}_u}) $$ 
is an injection for each $i\in\{0,1\}$, where $\overline{\widetilde{\psi}}_u$ is the mod-$\ell$ reduction of $\widetilde{\psi}_u$. 
\end{lemma}
\begin{proof}
The proof follows a similar line of arguments as in \cite[Theorem 3.3.2]{dhar2025jacquet}. For the notational convenience, we put $U_\ell=U_{1.5}^{F^\ell}$ and $U=U_{1.5}^F$. For $i\in\{0,1\}$, let $\varphi_i$ be the natural map
$$ \widehat{{\rm H}}^i(\sigma, \mathcal{L}(U_\ell,
\widetilde{\eta}_u))\rightarrow \widehat{{\rm H}}^i(\sigma, \mathcal{L}). $$ 
\vspace{0.1 cm}
\textbf{Claim.} The image of $\varphi_i$ is $\widehat{\rm H}^i(\sigma,\mathcal{L})(U,\overline{\widetilde{\psi}}_u)$.
\vspace{0.1 cm}\\
We consider the case $i=0$. Let $v$ be in the image of $\varphi_0$, and
let $\tilde{v}$ be a lift of $v$ in $\mathcal{L}(U_\ell,\widetilde{\eta}_u)^\sigma$.
Then we have
\begin{equation}\label{lift_sum_1}
\sum_{x\in U_\ell}\widetilde{\eta}_u^{-1}(x)\Pi(x)\tilde{v} = 0.
\end{equation}
Since $\sigma$ acts freely on $U_\ell\setminus U$, there exists a subset $\mathcal{U}\subseteq U_\ell\setminus U$ such that $U_\ell\setminus U$ is the disjoint union of $\sigma^i\mathcal{U}$ for $1\leq i\leq \ell$. Then, using the fact that $\tilde{v}$ is a $\sigma$-fixed vector, it follows from the identity (\ref{lift_sum_1}) that
\begin{align*}
\sum_{x\in U_\ell}\widetilde{\eta}_u^{-1}(x)\Pi(x) v 
&=\sum_{y\in U}\widetilde{\psi}_u^{-1}(y)\Pi(y) v + \sum_{x\in U_\ell\setminus U} \widetilde{\eta}_u^{-1}(x)\Pi(x)v\\
&=\sum_{y\in U}\widetilde{\psi}_u^{-1}(y)\Pi(y) v + \sum_{i=1}^\ell\sum_{z\in\mathcal{U}}\widetilde{\psi}_u^{-1}(\sigma^i(z))\Pi(\sigma^i(z))\tilde{v}\\
&=\sum_{y\in U}\widetilde{\psi}_u^{-1}(y)\Pi(y) v
+ {\rm Nr}\big(\sum_{z\in\mathcal{U}}\widetilde{\eta}_u^{-1}(z)\Pi(z)\tilde{v}\big).
\end{align*}
Since
$$ {\rm Nr}\big(\sum_{z\in\mathcal{U}}\widetilde{\eta}_u^{-1}(z)\Pi(z)\tilde{v}\big) = 0 $$ 
in $\widehat{{\rm H}}^0(\mathcal{L})$, it follows from the above identity that
$$ \sum_{y\in U} 
\overline{\widetilde{\psi}}_u^{-1}(y)\,\Pi(y)v = 0. $$ 
This implies that the vector $v$
belongs to
$\widehat{{\rm H}}^0(\mathcal{L})(U,\overline{\widetilde{\psi}}_u)$.
Conversely, let $w$ be an element of
$\widehat{{\rm H}}^0(\mathcal{L})(U,\overline{\widetilde{\psi}}_u)$.
Let $\tilde{w}$ be a lift of
$w$ in $\mathcal{L}^\sigma$. Consider the element
$$ \tilde{w}_1 = \tilde{w} - \lvert U_\ell\rvert \sum_{x\in U_\ell}
\widetilde{\eta}_u^{-1}(x)\Pi(x)\tilde{w}. $$
Then $\tilde{w}_1$ belongs to $\mathcal{L}(U_\ell,\widetilde{\eta}_u)^\sigma$ 
and $\varphi_0(\tilde{w}_1) = w$. 
Thus, we prove the claim for
$i=0$. The case $i=1$ is similar. Now, we consider the following short exact sequence of $\Lambda[\sigma]$-modules
$$ 0\longrightarrow \mathcal{L}(U_\ell,\widetilde{\eta}_u)\longrightarrow \mathcal{L}\longrightarrow \mathcal{L}_{U_\ell,\widetilde{\eta}_u}\longrightarrow 0. $$
The corresponding long exact sequence of Tate cohomology groups gives the following exact sequences:
$$ \widehat{\rm H}^0(\sigma,\mathcal{L}(U_\ell,\widetilde{\eta}_u))\xrightarrow {\varphi_0} \widehat{\rm H}^0(\sigma,\mathcal{L})\longrightarrow \widehat{\rm H}^0(\sigma,\mathcal{L}_{U_\ell,
\widetilde{\eta}_u}) $$
and 
$$ \widehat{\rm H}^1(\sigma,\mathcal{L}(U_\ell,\widetilde{\eta}_u))\xrightarrow {\varphi_1} \widehat{\rm H}^1(\sigma,\mathcal{L})\longrightarrow \widehat{\rm H}^1(\sigma,\mathcal{L}_{U_\ell,
\widetilde{\eta}_u}). $$
Since the image of $\varphi_i$ is $\widehat{\rm H}^i(\sigma,\mathcal{L})(U,\overline{\widetilde{\psi}}_u)$, the natural map 
$$ \widehat{\rm H}^i(\sigma,\mathcal{L})_{U,\overline{\widetilde{\psi}}_u}\rightarrow \widehat{\rm H}^i(\sigma,\mathcal{L}_{U_\ell,
\widetilde{\eta}_u}) $$
is injective. Hence, the lemma follows.
\end{proof}
\section{Tate cohomology of unipotent cuspidal representations}
In this section, we prove Theorem \ref{intro_main_thm}. We first prove the theorem for a $\sigma$-stable unipotent cuspidal representation of $G^F$, where $G$ is simple of adjoint type, and $p$ is good for $G$. Then we deduce it for general $G$ under the hypothesis on $\ell$. We begin with the discussion of Galois invariance of unipotent cuspidal representations.
\subsection{Galois invariance of unipotent cuspidal representations}
Recall that $G$ is a connected reductive group with a Frobenius map $F$ on $G$, and assume that $G$ is simple of adjoint type. Assume further that $p$ is good for $G$. Let $\Pi$ be a unipotent cuspidal $\mathcal{K}$-representation of $G^{F^\ell}$. Let $C=C_\Pi$ be the unipotent support of $\Pi$. It follows from \cite[Proposition 3.1]{Dudas_Malle_irr_mod_l} that the representation $\Pi$ occurs with multiplicity one in $\Gamma_C^{G^{F^\ell}}
\otimes_\Lambda\mathcal{K}$, i.e.,
\begin{equation}\label{mult_1_gggr}
{\rm dim}_{\mathcal{K}}\,{\rm Hom}_{G^{F^\ell}}(\Gamma_C^{G^{F^\ell}}\otimes_\Lambda \mathcal{K},\Pi) = 1,
\end{equation}
where $\Gamma_C^{G^{F^\ell}}$ is a direct sum of $\Gamma_u^{G^{F^\ell}}$'s where $u$ runs over a set of representatives of $G^{F^\ell}$-orbits in $C^{F^\ell}$. Let $\sigma$ be the Galois automorphism of $\mathbb{F}_{q^\ell}/\mathbb{F}_q$, given by $x\mapsto x^q$. Then the finite group $G^{F^\ell}$ is endowed with an action of $G^{F^\ell}$. For a representation $\rho$ of $G^{F^\ell}$, let $\rho^\sigma$ be the twisted representation, defined by $\rho^\sigma(g) = \rho(\sigma(g))$ for all $g\in G^{F^\ell}$. 
\begin{lemma}\label{sigma_stable_mult_1_gggr}
Let $\Pi$ be a unipotent cuspidal $\ell$-adic representation of $G^{F^\ell}$ such that $\Pi\simeq \Pi^\sigma$. Then there exists a $\sigma$-stable orbit $[u]$ in $C^{F^\ell}/G^{F^\ell}$ such that 
\begin{equation}\label{mult_1_cusp_unip}
{\rm dim}_{\mathcal{K}}\,{\rm Hom}_{G^{F^\ell}}(\Gamma_u^{G^{F^\ell}}\otimes_\Lambda\mathcal{K},\Pi) = 1.
\end{equation}
\end{lemma}
\begin{proof}
Consider the unipotent class $\sigma(C)$ of $G^{F^\ell}$. Since $C$ is the unipotent support of $\Pi$, the unipotent class $\sigma(C)$ is the unipotent support of $\Pi^\sigma$. Since $\Pi\simeq \Pi^\sigma$, we get that $\sigma(C)$ is also a unipotent support of $\Pi$. By uniqueness of $C$, we have
$\sigma(C) =C$, i.e., $C$ is $\sigma$-stable. From (\ref{mult_1_gggr}), we get the following identity
\begin{equation}\label{sigma_stable_mult_1}
\sum_{[u]\in C^{F^\ell}/G^{F^\ell}}\langle\gamma_u^{G^{F^\ell}},\chi_\Pi\rangle = 1.
\end{equation}
Let $[u]\in C^{F^\ell}/G^{F^\ell}$. Suppose $[u]$ is not $\sigma$-stable and $\langle\gamma_u^{G^{F^\ell}},\chi_\Pi\rangle=1$, then using the fact that $\chi_\Pi=\chi_\Pi^\sigma$, we get $\langle\gamma_{\sigma(u)}^{G^{F^\ell}},\chi_\Pi\rangle=1$, which contradicts to the identity (\ref{sigma_stable_mult_1}). Hence, there exists exactly one $\sigma$-stable orbit $[u]$ in $C^{F^\ell}/G^{F^\ell}$ such that
$$ \langle\gamma_u^{G^{F^\ell}},\chi_\Pi\rangle = 1. $$
Thus, we get the required lemma.
\end{proof}

\begin{lemma}\label{sigma_action_uni_cusp}
Let $\Pi$ be a unipotent cuspidal $\ell$-adic representation of $G^{F^\ell}$ such that $\Pi\simeq \Pi^\sigma$. Then $\Pi$ extends uniquely as a representation of $G^{F^\ell}\rtimes \sigma$.
\end{lemma}
\begin{proof}
By Lemma \ref{sigma_stable_mult_1_gggr}, $\Pi$ is realized as a subrepresentation of $\Gamma_u^{G^{F^\ell}}
\otimes_\Lambda\mathcal{K}$ for a $\sigma$-stable orbit $[u]\in C^{F^\ell}/G^{F^\ell}$. Denote by $\mathcal{W}(\Pi,\widetilde{\eta}_u)$ the realization of $\Pi$ in $\Gamma_u^{G^{F^\ell}}
\otimes_\Lambda\mathcal{K}$. Let $T_\sigma:\mathcal{W}(\Pi,\widetilde{\eta}_u)\rightarrow
\Gamma_u^{G^{F^\ell}}
\otimes_\Lambda\mathcal{K}$ be the linear map defined by $T_\sigma(f)(g) = f(\sigma^{-1}(g))$, for $g\in G^{F^\ell}$. Since the character $\widetilde{\eta}_u$ is $\sigma$-invariant, the image of $T_\sigma$ is equal to $\mathcal{W}(\Pi^\sigma,\widetilde{\eta}_u)$. Since $\Pi\simeq \Pi^\sigma$, it follows from the multiplicity one property of $\Pi$ (Lemma \ref{sigma_stable_mult_1_gggr}) that the image of $T_\sigma$ is $\mathcal{W}(\Pi,\widetilde{\eta}_u)$. Hence, we get the required extension in the lemma. 

For the uniqueness of the $\sigma$-action, let $T_\sigma$ and $T_\sigma'$ be the action of $\sigma$ on $\Pi$ for two different extensions of $\Pi$. Then, using Schur's lemma, we have  $T_\sigma=\lambda T_\sigma'$, for some
$\lambda\in \mathcal{K}$. Since $\lambda^\ell=1$ and $\mathcal{K}$ is the maximal unramified extension of $\mathbb{Q}_\ell$, we conclude that  $\lambda=1$. This completes the proof.
\end{proof}
\subsection{}\label{Descent_orbit_subsection}
Let $G$ be a connected reductive group defined over $\mathbb{F}_q$ and let $F$ be the corresponding Frobenius morphism on $G$. Let $C$ be an $F$-stable unipotent class in $G$. Let $C^F/G^F$ denote the set of distinct $G^F$-orbits in $C^F$. Note that $C^{F^\ell}/G^{F^\ell}$ is endowed with an action of $\sigma$, given by $\sigma\cdot[w]=[\sigma(w)]$ for all $[w]\in C^{F^\ell}/G^{F^\ell}$. We say that $G$ has property ``$\mathscr{D}$-orbit" if the following holds:
\begin{center}
$\mathscr{D}$-orbit : For every $F$-stable unipotent class $C$ of $G$, the natural map is a bijection 
$$ \iota_\sigma:C^F/G^F\longrightarrow (C^{F^\ell}/G^{F^\ell})^\sigma. $$

\end{center}
Let us describe the above bijection in terms of cohomology sets. Choose a unipotent element $u\in G^F$. For $w=gug^{-1}\in C^F$, we have $x=g^{-1}F(g)\in C_G(u)$. Let $\overline{x}$ be the image of $x$ in the component group $\pi_0(C_G(u))$. Note that $\overline{x}\in \pi_0(C_G(u))$ gives a $1$-cocycle. If ${\rm cl}(\overline{x})$ denotes its equivalence class in ${\rm H}^1(F,\pi_0(C_G(u)))$, then the following map
$$ C^F/G^F \xrightarrow{\sim} {\rm H}^1(F,\pi_0(C_G(u))) $$
$$ [w]\mapsto {\rm cl}(\overline{x}) $$
is a bijection. Since the action of $\sigma$ on $C^{F^\ell}/G^{F^\ell}$ corresponds to the natural $F$-action on the cohomology ${\rm H}^1(F^\ell,\pi_0(C_G(u)))$, we have a bijection
$$ (C^{F^\ell}/G^{F^\ell})^\sigma\xrightarrow{\sim} {\rm H}^1(F^\ell,\pi_0(C_G(u)))^F. $$
Moreover, we have the following commutative diagram
\begin{equation}\label{comm_diag}
\xymatrixrowsep{0.2in}
\xymatrixcolsep{0.15in}
\xymatrix{
		C^F/G^F\ar[dd]_{}  \ar[rr]^{\iota_\sigma}  && (C^{F^\ell}/G^{F^\ell})^\sigma \ar[dd]^{} \\\\
		{\rm H}^1(F,\pi_0(C_G(u))) \ar[rr]_{{\rm res}} && {\rm H}^1(F^\ell,\pi_0(C_G(u)))^F 
}
\end{equation}
where the map ${\rm res}:{\rm H}^1(F,\pi_0(C_G(u)))\longrightarrow {\rm H}^1(F^\ell,\pi_0(C_G(u)))^F$ is defined as
$$ {\rm res}({\rm cl}(a)):={\rm cl}(aF(a)\cdots F^{\ell-1}(a)). $$
If the component group $\pi_0(C_G(u))$ is abelian and $\ell$ does not divide $\lvert\pi_0(C_G(u))\rvert$, then ${\rm res}$ is a bijection, and hence it follows from (\ref{comm_diag}) that the natural map $\iota_\sigma$ is a bijection.
\begin{remark}\normalfont\label{classical_descent}
Many classical algebraic groups $G$ admit the property $\mathscr{D}$-orbit; a list is given below.
\begin{center}
\begin{tabular}{|c c|} 
 \hline 
 $G$ & $G^F$ \\ 
 \hline\hline
 ${\rm GL}_n$ & ${\rm GL}_n(\mathbb{F}_q)$ \\ 
 \hline
 ${\rm GL}_n$ & ${\rm U}_n(\mathbb{F}_q)$ \\
 \hline
 ${\rm Sp}_{2n}$ & ${\rm Sp}_{2n}(\mathbb{F}_q)$ \\
 \hline
 ${\rm SO}_{2n+1}$ & ${\rm SO}_{2n+1}(\mathbb{F}_q)$ \\
 \hline
 ${\rm SO}_{2n}$ & ${\rm SO}_{2n}(\mathbb{F}_q)$ \\ 
 \hline
 ${\rm Spin}_{n}, n\geq 5$ & ${\rm Spin}_{n}(\mathbb{F}_q)$ \\
 \hline
\end{tabular}
\end{center}
Let $G$ be one of the classical algebraic groups listed above. For any unipotent $u\in G^{F^\ell}$, the component group $\pi_0(C_G(u))=C_G(u)/C_G(u)^0$ is an abelian $2$-group. Since $\ell$ is an odd prime, $\ell$ does not divide $\lvert \pi_0(C_G(u))\rvert$. Hence, $\sigma$ acts trivially on $\pi_0(C_G(u))$, and we get the required bijection.
\end{remark}
\subsubsection{}
Let $G$ be a connected reductive group defined over $\mathbb{F}_q$ with the corresponding Frobenius map $F$, and assume that $G$ has the property $\mathscr{D}$-orbit. Let $\Pi$ be a unipotent cuspidal $\ell$-adic representation of $G^{F^\ell}$ such that $\Pi\simeq \Pi^\sigma$. Let $C$ be the unipotent support of $\Pi$, and let $[w]\in C^{F^\ell}/G^{F^\ell}$ be a $\sigma$-stable orbit such that (\ref{mult_1_cusp_unip}) holds. Let $u\in G^F$ be such that $u$ belongs to the conjugacy class $[w]$. Let $U_{1.5}$ be the $F$-stable subgroup of $G$ associated with $u$. Since $w$ is $G^{F^\ell}$-conjugate to $u$, we have 
\begin{equation}\label{mult_1_iden}
{\rm dim}_{\mathcal{K}}\,{\rm Hom}_{G^{F^\ell}}(\Gamma_u^{G^{F^\ell}}\otimes_\Lambda\mathcal{K},\Pi) =1.
\end{equation}
Let $\mathcal{L}$ be a $G^{F^\ell}\rtimes\sigma$-stable $\Lambda$-lattice in $\Pi$. It follows from the identity (\ref{mult_1_iden}) that the space of coinvariants $\mathcal{L}_{U_{1.5}^{F^\ell},
\widetilde{\eta}_u}$ is a free $\Lambda$-module of rank one. Note that $\sigma$ acts trivially on this module. Hence, we have
\begin{equation}\label{Tate_coinv_free}
\widehat{\rm H}^0(\sigma,\mathcal{L}_{U_{1.5}^{F^\ell},
\widetilde{\eta}_u}) \simeq \overline{\mathbb{F}}_\ell,\,\,\,\,\,\, \widehat{\rm H}^1(\sigma, \mathcal{L}_{U_{1.5}^{F^\ell},
\widetilde{\eta}_u}) = 0. 
\end{equation}
Using these observations, we first deduce the vanishing of the first Tate cohomology group of unipotent cuspidal representations.
\begin{theorem}\label{first_tate_trivial}
Let $G$ be a connected reductive group of adjoint type defined over $\mathbb{F}_q$ with corresponding Frobenius map $F$. Assume that $G$ is simple of adjoint type such that $G$ has property $\mathscr{D}$-orbit, and $p$ is good for $G$. Let $\Pi$ be a unipotent cuspidal $\ell$-adic representation of $G^{F^\ell}$. Then, for any $G^{F^\ell}\rtimes\sigma$-stable $\Lambda$-lattice $\mathcal{L}$ in $\Pi$, we have 
$$ \widehat{\rm H}^1(\sigma,\mathcal{L}) = 0. $$
\end{theorem}
\begin{proof}
Let $\mathcal{L}$ be a $\Lambda[G^{F^\ell}\rtimes\sigma]$-stable lattice in $\Pi$. Since the Tate cohomology group $\widehat{\rm H}^1(\sigma,\mathcal{L})$ is cuspidal (by Proposition \ref{Tate_cuspidal}), it follows from the Frobenius reciprocity that
\begin{equation}\label{first_tate_equ_1}
{\rm Hom}_{G^F}\big(R_L^G(\overline{\Gamma}_C^{L^F}),
\widehat{\rm H}^1(\sigma,\mathcal{L})\big) = 0,
\end{equation}
for all proper $1$-split Levi subgroups $L$ of $G$ and for all $F$-stable unipotent classes $C$ of $L$ which are cuspidal for $L_{\rm ad}$. 

Let $C_\Pi$ be the unipotent support of $\Pi$. It follows from \cite[Proposition 3.1]{Dudas_Malle_irr_mod_l} that 
$$ {\rm dim}_{\mathcal{K}}\,{\rm Hom}_{G^{F^\ell}}(\Gamma_{C_\Pi}^{G^{F^\ell}}\otimes_\Lambda\mathcal{K},\Pi) =1 $$
or equivalently,
$$ \sum_{[w]\in C_\Pi^{F^\ell}/G^{F^\ell}}{\rm dim}_{\mathcal{K}}\,{\rm Hom}_{G^{F^\ell}}(\Gamma_{C_\Pi}^{G^{F^\ell}}
\otimes_\Lambda\mathcal{K},\Pi) =1, $$
where $C_{\Pi}^{F^\ell}/G^{F^\ell}$ is the set of distinct $G^{F^\ell}$-orbits in $C_\Pi^{F^\ell}$. Note that $C_\Pi$ is $F$-stable. By Lemma \ref{sigma_stable_mult_1_gggr}, there exists a $\sigma$-stable orbit $[w]\in C^{F^\ell}/G^{F^\ell}$ such that
$$ {\rm dim}_{\mathcal{K}}\,{\rm Hom}_{G^F}(\Gamma_w^{G^{F^\ell}}\otimes_\Lambda\mathcal{K}, \Pi) = 1,$$
and we have
$$ {\rm dim}_{\mathcal{K}}\,{\rm Hom}_{G^F}(\Gamma_{w'}^{G^{F^\ell}}\otimes_\Lambda\mathcal{K}, \Pi) = 0, $$
for all $\sigma$-stable orbits $[w']\in C^{F^\ell}/G^{F^\ell}$ with $[w']\ne [w]$. Since $G$ has the property $\mathscr{D}$-orbit, there exist $u, u'\in G^F$ with $[u]\ne [u']$ in $C_\Pi^F/G^F$ such that $[u]=[w]$ and $[u']=[w']$ in $C_\Pi^{F^\ell}/G^{F^\ell}$. Now, using Frobenius reciprocity, we get
$$ {\rm dim}_{\mathcal{K}}(\Pi_{U_{1.5}^{F^\ell},\widetilde{\eta}_u}) = 1,\,\,\,\,\,\, {\rm dim}_{\mathcal{K}}(\Pi_{U_{1.5}'^{F^\ell},\widetilde{\eta}_{u'}}) = 0, $$
where $U_{1.5}'$ is the $F$-stable subgroup of $G$ associated with the unipotent element $u'$. Thus, the space of coinvariants $\mathcal{L}_{U_{1.5}^{F^\ell},
\widetilde{\eta}_u}$ is a free $\Lambda$-module of rank one, and $\mathcal{L}_{U_{1.5}'^{F^\ell},
\widetilde{\eta}_{u'}} = 0$ for $[u']\ne [u]$ in $C_\Pi^{F}/G^{F}$. Hence, 
$$ \widehat{\rm H}^1(\sigma,\mathcal{L}_{U_{1.5}^{F^\ell},
\widetilde{\eta}_u}) = 0,\,\,\,\,\,\, \widehat{\rm H}^1(\sigma,\mathcal{L}_{U_{1.5}'^{F^\ell},
\widetilde{\eta}_{u'}})=0. $$
Using Lemma \ref{Tate_injection}, we get that 
$\widehat{\rm H}^1(\sigma,\mathcal{L})_{U_{1.5}^F,\overline{\widetilde{\psi}}_u}$ is trivial for all $u\in C_\Pi^F$--which is equivalent to saying
\begin{equation}\label{first_Tate_equ_2}
{\rm Hom}_{G^F}(\overline{\Gamma}_{C_\Pi}^{G^F},\widehat{\rm H}^1(\sigma,\mathcal{L})) = \{0\}.
\end{equation}
For an $F$-stable unipotent class $D$ in $G$ such that $D\ne C_\Pi$, we have 
$$ {\rm dim}_{\mathcal{K}}\,{\rm Hom}_{G^{F^\ell}}(\Gamma_D^{G^{F^\ell}}\otimes_\Lambda
\mathcal{K},\Pi) =0. $$
Using a similar line arguments as in the case of $C_\Pi$, we get that $\widehat{\rm H}^1(\sigma,\mathcal{L})_{U_{1.5}^F,\overline{\widetilde{\psi}}_u}$ is trivial for all $u\in D^F$, and hence
\begin{equation}\label{first_Tate_equ_3}
{\rm Hom}_{G^F}(\overline{\Gamma}_D^{G^F},\widehat{\rm H}^1(\sigma,\mathcal{L})) = \{0\}.
\end{equation}
From (\ref{first_tate_equ_1}), (\ref{first_Tate_equ_2}) and (\ref{first_Tate_equ_3}), we obtain 
$$ {\rm Hom}_{G^F}(\mathscr{P}, \widehat{\rm H}^1(\sigma,\mathcal{L})) = \{0\}. $$
Since $\mathscr{P}$ is a progenerator, the functor ${\rm Hom}_{G^F}(\mathscr{P},\_)$ is faithful. Hence, the space $\widehat{\rm H}^1(\sigma,\mathcal{L})$ is trivial. 
\end{proof}
\subsubsection{}
Next, we prove the irreducibility of the Tate cohomology group $\widehat{\rm H}^0(\sigma,\mathcal{L})$. We start by proving the following lemma. 
\begin{lemma}\label{Pro_gen_1}
Let $G$ be as in Theorem \ref{first_tate_trivial}. Assume that $p$ is good for $G$. Let $\mathscr{P}$ be the progenerator as defined by (\ref{progenerator}). Then we have 
$$ {\rm dim}_{\overline
{\mathbb{F}}_\ell}\,{\rm Hom}_{G^F}
(\mathscr{P},\widehat{\rm H}^0(\sigma,
\mathcal{L}))\leq 1. $$
\end{lemma}
\begin{proof}
Since $\Pi$ is cuspidal, the Tate cohomology group $\widehat{\rm H}^0(\sigma,\mathcal{L})$ is also cuspidal by Proposition \ref{Tate_cuspidal}. Therefore, for any proper $1$-split Levi subgroup $L$ of $G$, we have $^\ast R_L^G(\widehat{\rm H}^0(\sigma,\mathcal{L}))=0$. Frobenius reciprocity gives
\begin{equation}\label{equ_5}
{\rm Hom}_{G^F}(R_L^G(\overline{\Gamma}_{C}^{L^F}),\widehat{\rm H}^0(\sigma,\mathcal{L})) \simeq 
{\rm Hom}_{L^F}\big(\overline{\Gamma}_{C}^{L^F},\,^\ast R_L^G(\widehat{\rm H}^0(\sigma,\mathcal{L}))\big) = \{0\},
\end{equation}
for all $F$-stable unipotent classes $C$ in $L$ that are cuspidal for $L_{\rm ad}$. Let $C_\Pi$ be the unipotent support of $\Pi$. Since $\Pi\simeq \Pi^\sigma$, uniqueness of $C_\Pi$ shows that $C_\Pi$ is $F$-stable. Note that 
$$ {\rm dim}_{\mathcal{K}}\,{\rm Hom}_{G^{F^\ell}}(\Gamma_{C_\Pi}^{G^{F^\ell}}\otimes_\Lambda\mathcal{K},\Pi) = 1. $$
By applying Lemma \ref{sigma_stable_mult_1_gggr} and the $\mathscr{D}$-orbit property of $G$, there exists a unipotent class $[u]$ in $C_\Pi^{F^\ell}/G^{F^\ell}$ with $\sigma(u)=u$ such that
$$ {\rm dim}_{\mathcal{K}}\,{\rm Hom}_{G^{F^\ell}}(\Gamma_u^{G^{F^\ell}}\otimes_\Lambda\mathcal{K},\Pi) = 1 $$
and 
$$ {\rm dim}_{\mathcal{K}}\,{\rm Hom}_{G^{F^\ell}}(\Gamma_{u'}^{G^{F^\ell}}\otimes_\Lambda\mathcal{K},\Pi) = 0 $$
for all $[u']\in (C_\Pi^{F^\ell}/G^{F^\ell})^\sigma$ with $\sigma(u')=u'$ and $[u]\ne [u']$. The above two identities shows that $\mathcal{L}_{U_{1.5}^{F^\ell},\widetilde{\eta}_u}$ is a free $\Lambda$-module of rank one, and $\mathcal{L}_{U_{1.5}'^{F^\ell},\widetilde{\eta}_{u'}}$ is trivial. Since $\sigma$ acts trivially on these modules, using (\ref{Tate_coinv_free}) and Lemma \ref{Tate_injection}, we get that
$$ {\rm dim}_{\overline{\mathbb{F}}_\ell} \widehat{\rm H}^0(\sigma,\mathcal{L})_{U_{1.5}^F,\overline{\widetilde{\psi}}_u} \leq 1, $$
for all $[u]$ in $C_\Pi^F/G^F$. Since $U_{1.5}^F$ is a $p$-group and $\ell\ne p$, we have ${\rm Hom}_{U_{1.5}^F}(\widetilde{\psi}_u,\widehat{\rm H}^0(\sigma,\mathcal{L}))\simeq \widehat{\rm H}^0(\sigma,\mathcal{L})_{U_{1.5}^F,\widetilde{\psi}_u}$. Then it follows from Frobenius reciprocity that
\begin{equation}\label{equ_1}
{\rm dim}_{\overline{\mathbb{F}}_\ell}\,{\rm Hom}_{G^F}(\overline{\Gamma}_{C_\Pi}^{G^F}, \widehat{\rm H}^0(\sigma,\mathcal{L}))\leq 1.
\end{equation}
Let $C'\ne C_\Pi$ be an $F$-stable cuspidal unipotent class in $G$. Since 
$$ {\rm dim}_{\mathcal{K}}\,{\rm Hom}_{G^{F^\ell}}(\Gamma_{C'}^{G^{F^\ell}}\otimes_\Lambda\mathcal{K},\Pi) = 0, $$
it follows from a similar argument as in the case of $C_\Pi$ that 
\begin{equation}\label{equ_2}
{\rm dim}_{\overline{\mathbb{F}}_\ell}\, {\rm Hom}_{G^{F}}(\overline{\Gamma}_{C'}^{G^F}, \widehat{\rm H}^0(\sigma,\mathcal{L})) = 0.
\end{equation}
Then it follows from (\ref{equ_5}), (\ref{equ_1}) and (\ref{equ_2}) that
$$ {\rm dim}_{\overline{\mathbb{F}}_\ell}\, {\rm Hom}_{G^F}(\mathscr{P},\widehat{\rm H}^0(\sigma,\mathcal{L}))\leq 1. $$
Hence, we get the required lemma.
\end{proof}
Now, we deduce the following theorem.
\begin{theorem}\label{Tate_irr}
Let $G$ be a connected reductive group with a Frobenius morphism $F$ on $G$ such that $G$ is simple of adjoint type. Assume that $G$ has the property $\mathscr{D}$-orbit and that $p$ is good for $G$. Let $\Pi$ be a cuspidal unipotent $\ell$-adic representation of $G^{F^\ell}$ such that $\Pi^\sigma\simeq \Pi$. Then, for any $\Lambda[G^{F^\ell}\rtimes\sigma]$-stable lattice $\mathcal{L}$ in $\Pi$, the Tate cohomology group $\widehat{\rm H}^0(\sigma,\mathcal{L})$ is irreducible as a representation of $G^F$.
\end{theorem}
\begin{proof}
Consider the the composition series of $\widehat{\rm H}^0(\sigma,\mathcal{L})$:
$$ 0=\xi_0\subseteq \xi_1\subseteq\cdots\subseteq \xi_r=\widehat{\rm H}^0(\sigma,\mathcal{L}). $$
By the exactness of the functor ${\rm Hom}_{G^F}(\mathscr{P},\_)$, we have
$$ {\rm dim}_{\overline{\mathbb{F}}_\ell}\,{\rm Hom}_{G^F}(\mathscr{P},\widehat{\rm H}^0(\sigma,\mathcal{L})) = \sum_{i=1}^r {\rm dim}_{\overline{\mathbb{F}}_\ell}\,{\rm Hom}_{G^F}(\mathscr{P},\xi_i/\xi_{i-1}). $$
Since each $\xi_i/\xi_{i-1}$ is irreducible, the space ${\rm Hom}_{G^F}(\mathscr{P},\xi_i/\xi_{i-1})$ is non-trivial. Hence, 
$$ {\rm dim}_{\overline{\mathbb{F}}_\ell}\,{\rm Hom}_{G^F}(\mathscr{P},\xi_i/\xi_{i-1})\geq 1, $$
for each $i$.
Since ${\rm dim}_{\overline{\mathbb{F}}_\ell}\, {\rm Hom}_{G^F}(\mathscr{P},\widehat{\rm H}^0(\sigma,\mathcal{L})) = 1$ (by Lemma \ref{Pro_gen_1}), it follows from the above identity that $r\leq 1$. Note that $\widehat{\rm H}^1(\sigma,\mathcal{L})$ is trivial (by Theorem \ref{first_tate_trivial}). Then, applying Lemma \ref{char_iden}, we get the following relation
$$ {\rm Tr}(g|_{\widehat{\rm H}^0(\sigma,\mathcal{L})}) = \overline{\chi_\Pi(g)} $$
for all $g\in G^F$. Since the mod-$\ell$ reduction $r_\ell(\Pi)$ is irreducible (see Theorem \ref{mod_l_irr_unip_cusp}), its trace is non-zero, and it follows from the above trace identity that $\widehat{\rm H}^0(\sigma,\mathcal{L})$ is non-trivial. This implies that $r=1$, and hence proving the $G^F$-irreducibility of $\widehat{\rm H}^0(\sigma,\mathcal{L})$.
\end{proof}

\subsection{General Case}
In this section, we prove the irreducibility of the zeroth Tate cohomology group of a $\sigma$-stable unipotent cuspidal representation of $G^{F^\ell}$, where $G$ is a connected reductive group with Frobenius $F$ on $G$ having the property $\mathscr{D}$-orbit, and $G$ is not necessarily simple of adjoint type. We primarily follow the idea of the proof of \cite[Theorem 3.2]{Dudas_Malle_irr_mod_l}.
\begin{theorem}\label{Tate_irr_general}
Let $G$ be a connected reductive group with a Frobenius map $F$ defining an $\mathbb{F}_q$-rational structure on $G$ such that $G$ has property $\mathscr{D}$-orbit. Assume that $\ell$ and $p$ are good for $G$, and $\ell$ does not divide $p\,\lvert Z(G)^F/Z^0(G)^F\rvert$. Let $\Pi$ be a unipotent cuspidal $\ell$-adic representation of $G^{F^\ell}$ such that $\Pi\simeq \Pi^\sigma$. Then, for any $\Lambda[G^{F^\ell}\rtimes\sigma]$-stable lattice $\mathcal{L}$ in $\Pi$, the Tate cohomology group $\widehat{\rm H}^0(\sigma,\mathcal{L})$ is irreducible. Moreover, the first Tate cohomology group $\widehat{\rm H}^0(\sigma,\mathcal{L})$ is trivial.
\end{theorem}
\begin{proof}
If we fix a regular embedding $G\hookrightarrow \widetilde{G}$ with the center $Z(\widetilde{G})$ being connected, then the restriction of unipotent representations of $\widetilde{G}^F$ to $G^F$ remains irreducible (\cite[Proposition 17.4]{Geck_basic_set_unip_blocks}). Furthermore, the assumption on $\ell$ shows that the unipotent parts of the decomposition matrices of $G^F$ and $\widetilde{G}^F$ are equal (see \cite{Cabanes_Enguehard}). Therefore, without loss of generality, we can assume that $Z(G)$ is connected.

Since $Z(G)^F = Z(G^F)$ and the unipotent characters of $G^F$ are trivial on the center, we are further reduced to the case where $G$ is semisimple of adjoint type. Then $G$ is a direct product
$$ G = G_1^{n_1}\times\cdots\times G_r^{n_r}, $$
where each $G_i$ is simple of adjoint type and $F$ cyclically permutes the copies of $G_i$ in $G_i^{n_i}$, i.e., $F(g_1,\cdots,g_{n_i}) = (F(g_{n_i}), F(g_{n_i-1}),\dots,F(g_1))$. For one such block, projection onto the first component gives a group isomorphism $(G_i^{n_i})^F\simeq G_i^{F^{n_i}}$.
This isomorphism carries unipotent representations to unipotent representations. The finite group $G^F$ is a direct product
$$ G^F = G_1^{F^{n_1}}\times\cdots\times G_r^{F^{n_r}}. $$
Since $\Pi$ is a unipotent cuspidal $\ell$-adic representation of $G^{F^\ell}$, we have
$$ \Pi = \Pi_1\times\cdots\times \Pi_r, $$
where each $\Pi_i$ is a unipotent cuspidal $\ell$-adic representation of $G_i^{F^{\ell n_i}}$ such that $\Pi_i\simeq \Pi_i^\sigma$. Since each $G_i$ simple of adjoint type, there exists a unique action of $\sigma$ on $\Pi_i$ (see Lemma \ref{sigma_action_uni_cusp}), and hence we get a unique $\sigma$-action on $\Pi$. Choose a $\Lambda[G_i^{F^{\ell n_i}}\rtimes\sigma]$-stable lattice $\mathcal{L}_i$ in $\Pi_i$, and we set 
$$ \mathcal{L}=\mathcal{L}_1\otimes
\cdots\otimes\mathcal{L}_r. $$
Then $\mathcal{L}$ is a $\Lambda[G^{F^\ell}\rtimes\sigma]$-stable lattice in $\Pi$. Since each $\widehat{\rm H}^1(\sigma,\mathcal{L}_i)$ is trivial by Theorem \ref{first_tate_trivial}, it follows from \cite[Corollary 2.2]{Modular_moonshine_3} that the first Tate cohomology group $\widehat{\rm H}^1(\sigma,\mathcal{L})$ is trivial, and 
$$ \widehat{\rm H}^0(\sigma,\mathcal{L}) = \widehat{\rm H}^0(\sigma,\mathcal{L}_1)\otimes
\cdots\otimes\widehat{\rm H}^0(\sigma,\mathcal{L}_r). $$
By Theorem \ref{Tate_irr}, each $\widehat{\rm H}^0(\sigma,\mathcal{L}_i)$ is irreducible. Hence, $\widehat{\rm H}^0(\sigma,\mathcal{L})$ is irreducible. Now, let us consider the short exact sequence of $\Lambda[\sigma]$-modules:
$$ 0\longrightarrow \mathcal{L}\xrightarrow{{\rm mult.} \ell}\mathcal{L}\longrightarrow\mathcal{L}\otimes_\Lambda\overline{\mathbb{F}}_\ell\longrightarrow 0. $$
The corresponding long exact sequence of Tate cohomology gives the following isomorphism
$$ \widehat{\rm H}^0(\sigma,\mathcal{L})\simeq \widehat{\rm H}^0(\sigma,\mathcal{L}\otimes_
\Lambda\overline{\mathbb{F}}_\ell). $$
Since the mod-$\ell$ reduction of $\Pi$ is irreducible (\cite[Theorem 3.2]{Dudas_Malle_irr_mod_l}), the Tate cohomology $\widehat{\rm H}^0(\sigma,\mathcal{L}\otimes_\Lambda\overline{\mathbb{F}}_\ell)$ is independent of the choice of $\mathcal{L}$ (see \cite[Proposition 6.1]{TV_functoriality}). Thus, we get the irreducibility of $\widehat{\rm H}^0(\sigma,\mathcal{L})$ for any $\Lambda[G^{F^\ell}\rtimes\sigma]$-stable lattice $\mathcal{L}$ in $\Pi$. This completes the proof.
\end{proof}
\begin{remark}\normalfont
Let $G$, $\ell$, $p$ be as in Theorem \ref{Tate_irr_general}. Let $\Pi$ is a unipotent cuspidal $\ell$-adic representation of $G^{F^\ell}$ such that $\Pi\simeq \Pi^\sigma$. Recall that the mod-$\ell$ reduction $r_\ell(\Pi)$ is irreducible and is $\sigma$-stable. Then it follows from the proof of Theorem \ref{Tate_irr_general} that the Tate cohomology group $\widehat{\rm H}^0(\sigma,r_\ell(\Pi))$ is irreducible.
\end{remark}

\bibliographystyle{amsalpha}
\bibliography{Tate_unipotent}
Sabyasachi Dhar,\\
Department of Mathematics, Indian Institute of Technology Bombay, Maharashtra-400076, India.\\
\texttt{mathsabya93@gmail.com},
\texttt{sabya@math.iitb.ac.in}

\end{document}